\documentclass[a4paper]{amsart}

\pdfoutput=1

\usepackage[utf8]{inputenc}
\usepackage[T1]{fontenc}
\usepackage{lmodern}
\usepackage{amsthm, amssymb, amsmath, amsfonts, mathrsfs}
\usepackage[body={15cm,21.5cm},centering]{geometry} 

\usepackage[colorlinks=true, pdfstartview=FitV, linkcolor=blue, citecolor=blue, urlcolor=blue,pagebackref=false]{hyperref}
\usepackage{microtype}

\newcommand{\no}{\nonumber}
\newcommand{\be}{\begin{equation}}
\newcommand{\ee}{\end{equation}}
\newcommand{\bi}{\begin{itemize}}
\newcommand{\ei}{\end{itemize}}
\newcommand{\br}{\begin{eqnarray}}
\newcommand{\er}{\end{eqnarray}}

\newcommand{\eps}{\varepsilon}
\renewcommand{\epsilon}{\varepsilon}

\newcommand{\norm}[1]{\lVert #1 \rVert}

\newcommand{\avg}[1]{\langle #1 \rangle}
\newcommand{\bigavg}[1]{\left \langle #1 \right \rangle}

\newcommand{\commentout}[1]{}

\newcommand{\B}{\mathbb{B}}
\renewcommand{\P}{\mathbb{P}}
\newcommand{\R}{\mathbb{R}}
\newcommand{\Z}{\mathbb{Z}}
\newcommand{\na}{\nabla}
\newcommand{\Ll}{\left}
\newcommand{\Rr}{\right}
\newcommand{\e}{\mathbf{e}}

\newcommand{\Ah}{A_{\mathsf{h}}}

\newcommand{\K}{\mathscr{K}_\xi}

\newcommand{\Qx}{\mathsf{Q}^{(\xi)}}
\newcommand{\ov}{\overline}
\newcommand{\un}{\underline}

\renewcommand{\d}{{\mathrm{d}}}

\newcommand{\xxi}{^{(\xi)}}

\newcommand{\cGh}{\mathcal{G}_{\mathsf{h}}}
\newcommand{\expec}[1]{\left\langle #1 \right\rangle}

\newcommand{\Cl}{\mathcal{C}^{\alpha}_{\mathrm{loc}}}
\newcommand{\var}{\mathbb{V}\mathrm{ar}}

\newtheorem{theo}{Theorem}[section]
\newtheorem{prop}[theo]{Proposition}

\newtheorem{lem}[theo]{Lemma}

\theoremstyle{definition}

\newtheorem{rmk}[theo]{Remark}

\newcommand{\Rm}{{\mathbb R}}
\newcommand{\Zm}{{\mathbb Z}}

\def\ZZint#1#2#3{{\setbox0=\hbox{$#1{#2#3}{\int}$}
\vcenter{\hbox{$#2#3$}}\kern-.5\wd0}}

\numberwithin{equation}{section}

\begin{document}

\title{Scaling limit of the corrector in stochastic homogenization}

\author[J.-C. Mourrat]{Jean-Christophe Mourrat}
\author[J. Nolen]{James Nolen}
\date{\today}
\address[Jean-Christophe Mourrat]{ Ecole normale supérieure de Lyon, CNRS, Lyon, France}
\email{jean-christophe.mourrat@ens-lyon.fr}
\address[James Nolen]{Department of Mathematics, Duke University, Durham, North Carolina, USA}
\email{nolen@math.duke.edu}
\begin{abstract}

In the homogenization of divergence-form equations with random coefficients, a central role is played by the corrector. We focus on a discrete space setting and on dimension $3$ and more. Completing the argument started in \cite{correl}, we identify the scaling limit of the corrector, which is akin to a Gaussian free field.

\bigskip


\medskip


\end{abstract}
\maketitle


%
%
%
%
%
%
%
\section{Introduction}
\label{s:intro}

We consider a random conductance problem on $\mathbb{Z}^d$ associated with the discrete divergence-form operator $\na^* A \na$, where the coefficients of $A$ are independent, identically distributed random variables, bounded away from $0$ and infinity.  The main object of this paper is the stationary random corrector $\phi$, satisfying $\na^* A (\xi + \na \phi^{(\xi)}) = 0$ with $\xi \in \Rm^d$ being a fixed vector, $d \geq 3$. This random function $\phi$ plays a central role \cite{Kun, Biskup, glotto} in homogenenization theory for the operator $\na^* A \na$, a discrete analogue of the random elliptic operators considered in \cite{Koz, PV1}.  Our main result is that for $d \geq 3$ the appropriately rescaled corrector converges to a Gaussian field that has the homogeneity of a Gaussian free field.

To make this statement precise, we need to introduce some notation. We view $\Z^d$ as a graph with edges between nearest-neighbors, and we denote by $\B = \{ (x,y) \in \Z^d \times \Z^d \;|\; |x - y| = 1\}$ the set of (non-oriented) edges.  Let $(\e_1,\ldots,\e_d)$ be the canonical basis of $\Z^d$. For every edge $e \in \B$, there exists a unique pair $(\un{e},i) \in \Z^d \times \{1,\ldots d\}$ such that $e$ links $\un{e}$ to $\un{e} + \e_i$. Given such a pair, we write $\ov{e} = \un{e} + \e_i$. We call $\un{e}$ the \emph{base point} of the edge~$e$. For $f : \Z^d \to \R$, we let $\na f : \B \to \R$ be the gradient of $f$, defined by
$$
\na f(e) = f(\ov{e}) - f(\un{e}).
$$
We write $\na^*$ for the formal adjoint of $\na$, that is, for $F : \B \to \R$, $\na^*F : \Z^d \to \R$ is defined via
$$
(\na^* F)(x) = \sum_{i = 1}^d F((x-\e_i,x)) - F((x,x+\e_i)).
$$
Given a family $\{ a(e) \}_{e\in \B}$ of positive real numbers and a function $F : \B \to \R$, we let $A F(e) = a(e) F(e)$. This provides us with a precise definition of the operator $\na^* A \na$. 

In order to facilitate the derivation of the result (and match the assumptions made in \cite{correl} for the same reason), we make the simplifying assumption that the coefficients $\{ a(e) \}_{e \in \B}$ are constructed as follows: we give ourselves a family $\{ \zeta_e \}_{e \in \B}$ of independent standard Gaussian random variables, and define $a(e) := {\bf a}(\zeta_e)$, where ${\bf a}:\R \to \R$ is a twice differentiable function with bounded first and second derivatives, and taking values in a compact subset of $(0,\infty)$. We denote by $(\Omega,\P)$ the underlying probability space and probability measure, and by $\langle \cdot \rangle$ the associated expectation.

Gloria and Otto showed in \cite{glotto} that for $d \ge 3$, for every given $\xi \in \R^d$, there exists a random, stationary $\phi\xxi : \Z^d \to \R$, called the \emph{corrector}, satisfying
\begin{equation}
\label{e:def:corr}
\na^* A(\xi + \na \phi\xxi) = 0.
\end{equation}
Also, $\avg{|\phi\xxi(0)|^p} < \infty$ for all $p \geq 1$. If we interpret the random variables $\{a(e)\}_{e \in \B}$ as conductances, the function $x \cdot \xi + \phi\xxi(x)$ represents a potential with macroscopic gradient $\xi$, and the value $a(e) (\xi + \na \phi\xxi)(e)$ represents a current across edge $e$.  The effective conductivity of the network is the matrix $\Ah$ defined by
\[
\xi \cdot \Ah \xi = \sum_{e = (0,\e_i) } \langle (\xi + \na \phi\xxi)(e) a(e) (\xi + \na \phi\xxi)(e) \rangle.
\] 
Let $\cGh$ be the Green function of the continuum differential operator $-\na \cdot \Ah \na$. The main result of \cite{correl} is that there exists a $d \times d$ matrix $\Qx$ and a constant $C < \infty$ such that letting
$$
\K(x) := \int_{\Rm^d} \na \cGh(y) \cdot \Qx \, \na \cGh(x-y) \ \d y,
$$
we have
\begin{equation}
\label{e:struct}
\Ll| \expec{\phi\xxi(0) \ \phi\xxi(x)} - \K(x) \Rr| \le C \ \frac{1 \vee \log^2 |x|}{|x|^{d-1}} \qquad (x \in \Z^d).
\end{equation}
The matrix $\Qx$ is positive semi-definite, and $\Qx_{j,k}$ is defined in terms of $\nabla \phi\xxi$, $\nabla \phi^{({\bf e}_j)}$ and $\nabla \phi^{({\bf e}_k)}$.  We refer to \cite{correl} for a more precise description of the matrix $\Qx$. The function $\K(x)$ has the same homogeneity as the Green function: $\K(\lambda x) = \lambda^{2 - d} \K(x)$ for all $x \in \Rm^d$, $\lambda > 0$.

In the present article, we show that $\phi\xxi$ converges to a Gaussian field.  Following \cite{besov}, for any $\alpha < 0$, we denote by $\Cl = \Cl(\R^d)$ the (separable) local Hölder space with exponent of regularity~$\alpha$. The following theorem is our main result:


\begin{theo}[Scaling limit of the corrector]
\label{t:main} Recall that we assume $d \ge 3$, and let $\Phi\xxi_\eps$ be the (random) distribution defined by
\begin{equation}
\label{e:def:phieps}
\Phi\xxi_\eps(f) := \epsilon^{\frac{d}{2} + 1} \sum_{x \in \Zm^d} f(\epsilon x) \phi\xxi(x)
\end{equation}
for $f \in C^{\infty}_c$. For every $\alpha < 1-d/2$, the distribution $\Phi\xxi_\eps$ converges in law to $\ov \Phi\xxi$ as $\eps \to 0$ in the topology of $\Cl$, where $\ov \Phi\xxi$ is the Gaussian random field such that for every smooth, compactly supported function $f$, $\Phi\xxi(f)$ is a centred Gaussian with variance
\begin{equation}
\label{e:def:sigma}
\sigma^2(f) := \int f(x) \K(y-x) f(y) \, \d x \d y.
\end{equation}
\end{theo}

In \cite{BB07} (Conjecture 5), it was conjectured that the appropriate scaling limit for the corrector is a Gaussian free field. While Theorem \ref{t:main} shows that the limit is a Gaussian field, the covariance structure can be different from the Gaussian free field, as explained in \cite{correl}.

From \eqref{e:struct}, one infers that for every smooth and compactly supported function $f$, one has
$$
\var\Ll[\Phi\xxi_\eps(f)\Rr] \xrightarrow[\eps \to 0]{} \sigma^2(f). 
$$
In order to prove Theorem~\ref{t:main}, it suffices to prove that (1) the fluctuations of $\Phi\xxi_\eps(f)$ are Gaussian, and that (2) the law of $\Phi\xxi_\eps$ is tight in $\Cl$. For the first part, we will in fact be more precise and give a rate of convergence:
\begin{prop}[Gaussian fluctuations]
\label{p:gauss}
Let $d_K$ denote the Kantorovich--Wasser\-stein distance
\[
d_K(X,W) = \sup  \{ \langle h(X) - h(W) \rangle \;|\; \norm{h'}_\infty  \leq 1 \},
\]
and let $Y$ be a standard Gaussian random variable. For every $f \in C^\infty_c$, we have
\be
\label{normapprox1}
d_{K}( \Phi_{\epsilon}(f), \sigma(f) Y) \xrightarrow[\eps \to 0]{} 0.
\ee
Moreover, if $\sigma(f) > 0$ and $\sigma_\epsilon = \sqrt{\var(\Phi_\epsilon(f))}$, then 
\be
d_{K}(\sigma_\epsilon^{-1} \Phi_\epsilon(f),  Y) \lesssim \epsilon^{d/2} \log|\epsilon| 
\label{normapprox2}
\ee
as $\epsilon \to 0$.
\end{prop}

Proposition \ref{p:gauss} is proved in Section \ref{sec:gauss}.  We then prove the following tightness result in Section~\ref{sec:tight}.

\begin{prop}[Tightness]
\label{p:tight}
For every $\alpha < 1- d/2$, the sequence of (random) distributions $\Ll(\Phi\xxi_\eps\Rr)_{\eps \in (0,1]}$ is tight in $\Cl$.
\end{prop}

\begin{rmk}
Theorem~\ref{t:main} can be reformulated as the joint convergence in law of $(\Phi^{(\e_1)}_\eps, \ldots, \Phi^{(\e_d)}_\eps)$ to a Gaussian vector field. Indeed, tightness for the product topology follows from the tightness of each of the coordinates. The limit law is then uniquely identified since Theorem~\ref{t:main} gives a characterisation of the limit law of every linear combination of $(\Phi^{(\e_1)}_\eps, \ldots, \Phi^{(\e_d)}_\eps)$. The covariance structure of the limiting field can be inferred by a polarization (with respect to $\xi$) of the left-hand side of \eqref{e:struct}.
\end{rmk}


\section{Proof of Proposition \ref{p:gauss}} \label{sec:gauss}

From now on, we drop the dependence on $\xi$ in the notation for simplicity, writing $\phi$ and $\Phi_\eps$ instead of $\phi\xxi$ and $\Phi\xxi_\eps$, respectively.

If $\sigma(f) = 0$, then the distribution of $\sigma(f)W$ is $\delta_0$. In this case, (\ref{normapprox1}) follows immediately from Chebyshev's inequality and the fact that $\avg{ \Phi_\epsilon(f)} = 0$:  for all Lipschitz functions $h$ with $\norm{h'} \leq 1$,
\br
|\avg{ h(\Phi_{\epsilon}(f))} - \avg{h(\sigma(f) W)}|  & = & |\avg{ h(\Phi_{\epsilon}(f))} - h(0) | \no \\
& \leq & \avg{ |h(\Phi_{\epsilon}(f)) - h(0)|} \no \\
& \leq & \norm{h'} \avg{ | \Phi_{\epsilon}(f)| } \leq \sqrt{\var(\Phi_{\epsilon}(f))}
\er
Hence,
\[
d_{K}(\Phi_{\epsilon}(f), \sigma(f)W) \leq \sqrt{\var(\Phi_{\epsilon}(f))} \to 0
\]
holds in this case.

So let us suppose that $\sigma(f) > 0$. We wish to prove (\ref{normapprox1}) and (\ref{normapprox2}).  Our proof will be based on the following proposition, which is a version of Theorem 2.2 in \cite{Ch3} and Theorem 3.1 (and Remark~3.6) of \cite{NP1}, stated in a form that is convenient for our purpose.

For a random variable $F \in L^2(\Omega)$ we say that $U = \partial_e F \in L^2(\Omega)$ is the weak derivative with respect to $\zeta_e$ (recall that $\zeta \sim N(0,1)$) if the following holds : for any finite subset $\Lambda \subset \B$ and any smooth, compactly supported function $\eta:\Rm^{|\Lambda|} \to \Rm$, we have
\[
\langle U \eta(\zeta) \rangle = \langle  F \zeta_e  \eta(\zeta) \rangle - \langle F \frac{\partial \eta}{\partial \zeta_e}(\zeta) \rangle,
\]
where $\eta(\zeta)$ depends only on $\{\zeta_{e'}\}_{e' \in \Lambda}$.

\begin{prop} \label{prop:stein}
Let $F  \in L^2(\Omega)$ be such that $\avg{F} = 0$, $\avg{F^2} = 1$. Assume also that $F$ has weak derivatives satisfying $\sum_{e} \avg{|\partial_e F|^4}^{1/2} < \infty$ and $\avg{|\partial_{e'} \partial_e F|^4} < \infty$ for all $e,e' \in \mathbb{B}$. Let $Y \sim N(0,1)$. Then
\be
\sup_{\norm{h'}_\infty \leq 1} \avg{h(F) - h(Y)}  \leq  \sqrt{\frac{5}{\pi}} \sqrt{\sum_{e'} \left( \sum_{e} \left \langle |\partial_e F |^4 \right \rangle^{1/4} \left \langle | \partial_{e'} \partial_e F|^4 \right \rangle^{1/4} \right)^2}. \label{steinsum}
\ee
\end{prop}

A proof of Proposition \ref{prop:stein} is given later in Section \ref{sec:stein}. We note that except for the numerical constant, the same result holds if the weak derivatives are replaced by the so-called Glauber derivatives, hence providing a version of the result that applies to functions of independent random variables that are not necessarily Gaussian, see Remark~2.3 of \cite{correl}.

\vspace{0.2in}

We will apply this proposition to $F = \sigma_\epsilon^{-1} \Phi_\eps(f)$, where $\sigma_\epsilon = \sqrt{\var(\Phi_\eps(f))} \to \sigma(f) > 0$. Without loss of generality, we may assume that $f(x) = 0$ when $|x| > 1$. For any edges $e, e'$,
\[
\partial_e \Phi =  \epsilon^{\frac{d}{2} + 1}\sum_{x \in \Zm^d} f(\epsilon x) \partial_e \phi(x), 
\]
and 
\[
\partial_{e'} \partial_e \Phi =  \epsilon^{\frac{d}{2} + 1} \sum_{x \in \Zm^d} f(\epsilon x) \partial_{e'} \partial_e \phi(x). 
\]
We recall from \cite{glotto} that the stationary function $\phi$ is defined by a limit 
\[
\phi(0) = \lim_{\mu \to 0} \phi_\mu(0), \quad \text{in}\;\;L^p(\Omega)
\]
for every $p \geq 1$, where the stationary process $\phi_\mu(x)$ satisfies the regularized corrector problem
\[
-\mu \phi_\mu + \na^* A(\xi + \na \phi_\mu) = 0
\]
with $\mu > 0$. The random variable $\phi_\mu(0)$, with $\mu > 0$, may be regarded as a function of the random variables $\{\zeta_e\}$ (as a Borel measureable function on $\Rm^\B$), and in \cite{glotto} the weak derivatives $\partial_e \phi_\mu$ were shown to exist, coinciding with $\frac{\partial \phi_\mu}{\partial \zeta_e}$.  Although the limit $\phi(0) \in L^p(\Omega)$ is not defined as a function of the $\{\zeta_e\}$ (it is only defined on a set of probability one), the weak derivatives $\partial_e \phi$ and $\partial_{e} \partial_{e'} \phi$ may be shown to exist through the limit as $\mu \to 0$, in the usual way. (Refer to  Remark 4.8 of \cite{correl} for more about this point.)

For $e = [z,z+{\bf e_i}]$, $\partial_e \phi$ satisfies
\be
- \nabla^* \cdot a \nabla (\partial_e \phi) = \nabla^* \cdot (\partial_e a) (\nabla \phi + \xi) \label{DjphiPDE}
\ee
and thus
\[
\partial_e \phi(x) = - (\partial_e a(e)) \nabla_{e} G(x,e)(\nabla \phi(e) + \xi_e),
\]
where $G$ denotes the Green function associated with the operator $\na^* A \na$. We use $\nabla_e G(x,e)$ to denote $G(x,z +{\bf e_i}) - G(x,z)$ where $e = [z,z+{\bf e_i}]$. Similarly, if $e' = [y,y + {\bf e_j}]$, then $\partial_{e'} \partial_e \phi$ satisfies
\be
- \nabla^* \cdot a \nabla (\partial_{e'} \partial_e \phi) = \nabla^* \cdot (\partial_{e'} \partial_e a) (\nabla \phi + \xi) + \nabla^* \cdot (\partial_e a(e)) (\nabla \partial_{e'} \phi) +  \nabla^* \cdot (\partial_{e'} a(e)) (\nabla \partial_{e} \phi). 
\ee
For $e \neq e'$, $\partial_{e'} \partial_e a = 0$, and we have
\br
\partial_{e'} \partial_e \phi(x) & = & - (\partial_e a(e)) \nabla_{e} G(x,e)(\nabla_{e} \partial_{e'} \phi(z) )  - (\partial_{e'} a(e')) \nabla_{e'} G(x,e')(\nabla_{e'} \partial_{e} \phi(y) ) \no \\
& = &  (\partial_e a(e))(\partial_{e'} a(e')) \nabla_{e} G(x,e)( \nabla_{e}  \nabla_{e'} G(e,e')(\nabla \phi(e') + \xi_{e'}) ) \no \\
& & + (\partial_e a(e))(\partial_{e'} a(e')) \nabla_{e'} G(x,e')( \nabla_{e'}  \nabla_{e} G(e',e)(\nabla \phi(e) + \xi_e) ),
\er
while for $e' = e$, we have
\br
\partial_e^2 \phi(x) & = &  2 (\partial_e a(e))^2  \nabla_{e} G(x,e)( \nabla_{e}  \nabla_{e} G(e,e)(\nabla_{e} \phi(e) + \xi_e) ) \no \\
&  &  -  (\partial_e^2 a(e))  \nabla_{e} G(x,e)( \nabla_{e} \phi(e) + \xi_e ). 
\er
Therefore, applying the generalized H\"older inequality and the fact that $|\partial_e a| = |{\bf a}'(\zeta_e)| \lesssim 1$ and $|\partial_e^2 a| = |{\bf a}''(\zeta_e)| \lesssim 1$
\br
\bigavg{ |\partial_e \partial_{e'} \phi(x)|^p }^{1/p} & \lesssim & \bigavg{ | \nabla_{e} G(x,e)|^{3p} }^{\frac{1}{3p}} \bigavg{| \nabla_{e} \nabla_{e'} G(e,e')|^{3p}}^{\frac{1}{3p}} \bigavg{ |\nabla \phi(e') + \xi|^{3p}}^{\frac{1}{3p}}  \no \\
& & + \bigavg{ | \nabla_{e} G(x,e')|^{3p} }^{\frac{1}{3p}} \bigavg{| \nabla_{e'} \nabla_{e} G(e',e)|^{3p}}^{\frac{1}{3p}} \bigavg{ |\nabla \phi(e) + \xi|^{3p}}^{\frac{1}{3p}}.  \label{dedephia}
\er
Also,
\br
\bigavg{ |\partial_e \phi(x)|^p }^{1/p} & \lesssim & \bigavg{ | \nabla_{e} G(x,e)|^{2p} }^{\frac{1}{2p}} \bigavg{ |\nabla \phi(x) + \xi|^{2p}}^{\frac{1}{2p}}. \label{dephib}
\er

Gloria and Otto \cite{glotto} proved that for all $p \geq 1$,
\[
\avg{ |\nabla \phi(e) + \xi|^p}^{1/p} \lesssim 1.
\]
Marahrens and Otto \cite{MO1} proved that the Green function satisfies
\[
\bigavg{|\nabla_e G(0,e)|^p}^{1/p} \lesssim \frac{1}{(1 + |e|)^{d - 1}}, 
\]
and
\[
\bigavg{|\nabla_{e'} \nabla_e G(e',e)|^p}^{1/p} \lesssim \frac{1}{(1 + |e - e'|)^{d}}
\]
for all $p \geq 2$.  By combining these crucial estimates with (\ref{dedephia}) and (\ref{dephi}), we obtain
\br
\bigavg{ |\partial_e \partial_{e'} \phi(x)|^p }^{1/p} & \lesssim & \left( \frac{1}{(1 + |x - e|)^{d - 1}}  +  \frac{1}{(1 + |x - e'|)^{d - 1}} \right)\cdot \frac{1}{(1 + |e - e'|)^{d}}, \label{dedephi}
\er
and
\br
\bigavg{ |\partial_e \phi(x)|^p }^{1/p} & \lesssim & \frac{1}{(1 + |x - e|)^{d - 1}}. \label{dephi}
\er
Applying (\ref{dephi}) and (\ref{dedephi}) to $\partial_e \Phi$ and $\partial_{e'} \partial_e \Phi$, and recalling that $f(x) = 0$ when $|x| > 1$, we have
\br
\epsilon^{- \frac{d}{2} - 1} \bigavg{|\partial_e \Phi|^p}^{1/p}  & \lesssim & \sum_{|x| \leq \epsilon^{-1}} \bigavg{  | \partial_e \phi(x))|^p}^{1/p} \no \\
& \lesssim & \sum_{|x| \leq \epsilon^{-1}} \frac{1}{(1 + |x - e|)^{d - 1}},  \label{dGammap}
\er
and
\br
\epsilon^{- \frac{d}{2} - 1}  \bigavg{|\partial_e \partial_{e'} \Phi|^p}^{1/p}  & \lesssim & \sum_{|x| \leq \epsilon^{-1}} \bigavg{  | \partial_e \partial_{e'} \phi(x))|^p}^{1/p}   \no \\
&  \lesssim &  \sum_{|x| \leq \epsilon^{-1}} \left( \frac{1}{(1 + |x - e|)^{d - 1}}  +  \frac{1}{(1 + |x - e'|)^{d - 1}} \right)\cdot \frac{1}{(1 + |e - e'|)^{d}}. \label{dGammap2}
\er

Now we are prepared to apply Proposition \ref{prop:stein}. The fact that $\sum_{e} \avg{|\partial_e \Phi|^4}^{1/2} < \infty$ and $\avg{|\partial_{e'} \partial_e \Phi|^4} < \infty$ follows from (\ref{dGammap}), (\ref{dGammap2}), and the following lemma:

\begin{lem} \label{lem:xesum}
For all $e \in \mathbb{Z}^d$ and $\epsilon \in (0,1]$.
\br
 \sum_{|x| \leq \epsilon^{-1}} \frac{1}{(1 + |x - e|)^{d - 1}}  \lesssim \frac{\epsilon^{-1}}{(1 + |\epsilon e|)^{d-1}}. \label{xesum}
\er

\end{lem}

By the estimates above, we have
\br
&&   \epsilon^{- d - 2} \sum_{e} \bigavg{| \partial_{e} \partial_{e'} \Phi|^4}^{1/4} \bigavg{|\partial_e \Phi|^4}^{1/4}   \no \\
& & \quad   \leq \sum_{e}  \left( \sum_{|x| \leq \epsilon^{-1}} \left( \frac{1}{(1 + |x - e|)^{d - 1}}  +  \frac{1}{(1 + |x - e'|)^{d - 1}} \right)\cdot \frac{1}{(1 + |e - e'|)^{d}} \right) \left(   \sum_{|x| \leq \epsilon^{-1}} \frac{1}{(1 + |x - e|)^{d - 1}}\right)    \no \\
& & \quad   = \sum_{e} \frac{1}{(1 + |e - e'|)^{d}}  \left( \sum_{|x| \leq \epsilon^{-1}} \left( \frac{1}{(1 + |x - e|)^{d - 1}}  +  \frac{1}{(1 + |x - e'|)^{d - 1}} \right) \right) \left(   \sum_{|x| \leq \epsilon^{-1}} \frac{1}{(1 + |x - e|)^{d - 1}}\right). \no
\er

\vspace{0.2in}

Therefore,
\br
&& \epsilon^{- d - 2} \sum_{e} \bigavg{| \partial_{e} \partial_{e'} \Phi|^4}^{1/4} \bigavg{|\partial_e \Phi|^4}^{1/4} \no \\
&& \quad \quad \quad  \leq \epsilon^{-2} \sum_{e} \frac{1}{(1 + |e - e'|)^{d}} \left( \frac{1}{(1 + |\epsilon e|)^{d-1}} + \frac{1}{(1 + |\epsilon e'|)^{d-1}}  \right) \frac{1}{(1 + |\epsilon e|)^{d-1}}. \label{d2phisum}
\er

\vspace{0.2in}

\begin{lem} \label{lem:eepsum}
Let $p > 0$. For all $e' \in \mathbb{Z}^d$ and $\epsilon \in (0,1/2]$,
\br
\sum_{e} \frac{1}{(1 + |e - e'|)^{d}} \frac{1}{(1 + |\epsilon e|)^p} \lesssim \frac{|\log \epsilon|}{(1 + |\epsilon e'|)^d} + \frac{  |\log \epsilon| }{(1 + |\epsilon e'|)^p}. \label{eepsum}
\er
\end{lem}

\vspace{0.2in}

Using Lemma \ref{lem:eepsum}, we conclude that
\br
\sum_{e} \frac{1}{(1 + |e - e'|)^{d}} \frac{1}{(1 + |\epsilon e|)^{2(d-1)}} & \lesssim& \frac{|\log \epsilon|}{(1 + |\epsilon e'|)^d}  + \frac{|\log \epsilon|}{(1 + |\epsilon e'|)^{2(d-1)}} \lesssim  \frac{|\log \epsilon|}{(1 + |\epsilon e'|)^d} \no
\er
and
\br
\sum_{e} \frac{1}{(1 + |e - e'|)^{d}}  \frac{1}{(1 + |\epsilon e|)^{d-1}} & \lesssim &  \frac{|\log \epsilon|}{(1 + |\epsilon e'|)^d}  + \frac{|\log \epsilon|}{(1 + |\epsilon e'|)^{d-1}} \lesssim \frac{|\log \epsilon|}{(1 + |\epsilon e'|)^{d-1}}. \no
\er
These bounds combined with (\ref{d2phisum}) imply that
\br
 \epsilon^{- 2d - 4}  \sum_{e'}  \left( \sum_{e} \bigavg{| \partial_{e} \partial_{e'} \Phi|^4}^{1/4} \bigavg{|\partial_e \Phi|^4}^{1/4} \right)^2  & \lesssim & \epsilon^{-4} \sum_{e'}  \left( \frac{|\log \epsilon|}{(1 + |\epsilon e'|)^d}  \right)^2 \no \\
& & + \epsilon^{-4} \sum_{e'} \frac{1}{(1 + |\epsilon e'|)^{2(d-1)}} \left( \frac{|\log \epsilon|}{(1 + |\epsilon e'|)^{d-1}} \right)^2 \no \\
&  \lesssim & \epsilon^{-4-d}(\log \epsilon)^2. 
\er

In view of this estimate, we see that the random variable $F_\epsilon = \sigma_\epsilon^{-1} \Phi_\epsilon$ satisfies
\[
 \sum_{e'}  \left( \sum_{e} \bigavg{| \partial_{e} \partial_{e'} F_\epsilon|^4}^{1/4} \bigavg{|\partial_e F_\epsilon|^4}^{1/4} \right)^2  \lesssim \sigma_\epsilon^{-4} \epsilon^d (\log \epsilon)^2.
\]
By Proposition \ref{prop:stein}, we conclude that
\[
\sup_{\norm{h'}_\infty \leq 1} \avg{h(F_\epsilon) - h(Y)} \lesssim  \sigma_\epsilon^{-2} \epsilon^{d/2} \log \epsilon  \lesssim  \epsilon^{d/2} \log \epsilon
\]
as $\epsilon \to 0$, since $\lim_{\epsilon \to 0} \sigma_{\epsilon} = \sigma > 0$ in this case.  This proves (\ref{normapprox2}).  Finally,
\br
\avg{h(\Phi_\epsilon) - h(\sigma W)} & = & \avg{h( \Phi_\epsilon) - h(\sigma_\epsilon W)} + \avg{ h(\sigma_\epsilon W)- h(\sigma W)} \no \\
& = & \sigma_\epsilon \avg{\hat h_\epsilon(\sigma_\epsilon^{-1} \Phi_\epsilon) - \hat h_\epsilon( W)} + \avg{ h(\sigma_\epsilon W)- h(\sigma W)} \no \\
& = & \sigma_\epsilon \avg{\hat h_\epsilon(F_\epsilon) - \hat h_\epsilon(W)} + \avg{ h(\sigma_\epsilon W)- h(\sigma W)}
\er
where $\hat h_{\epsilon}(\cdot ) = \sigma_{\epsilon}^{-1} h(\sigma_{\epsilon} \cdot)$ also satisfies $\norm{\hat h_\epsilon'}_\infty \leq 1$. Hence
\[
\sup_{\norm{h'}_\infty \leq 1} \avg{h(\Phi_\epsilon) - h(\sigma W)} \lesssim \epsilon^{d/2} \log \epsilon + |\sigma_\epsilon - \sigma|, 
\]
which implies that (\ref{normapprox1}) holds, as well.
\vspace{0.5in}

\commentout{

Next, we estimate $\kappa_0$.  By (\ref{dGammap}) and (\ref{xesum}), we see that
\[
\epsilon^{-d - 2} \kappa_0 \leq  \epsilon^{-2} \left( \sum_e \left( \frac{1}{(1 + |\epsilon e|)^{d-1} } \right)^4 \right)^{1/2}.
\]
Now
\br
\sum_e \left( \frac{1}{(1 + |\epsilon e|)^{d-1} } \right)^4  & = & \sum_{|e| < \epsilon^{-1}} \left( \frac{1}{(1 + |\epsilon e|)^{d-1} } \right)^4  + \sum_{|e| \geq \epsilon^{-1}}  \left( \frac{1}{(1 + |\epsilon e|)^{d-1} } \right)^4 \no \\
& \lesssim & \epsilon^{-d} + \sum_{|e| \geq \epsilon^{-1}}  \left( \frac{1}{|\epsilon e|^{d-1} } \right)^4 \lesssim \epsilon^{-d}
\er
Hence $\kappa_0 \lesssim \epsilon^{\frac{d}{2}}$.

Combining these estimates with the fact that
\[
\lim_{\epsilon \to 0} \var(\Phi_\epsilon(f)) = \sigma(f) > 0,
\]
we conclude that
\[
\frac{\kappa_0}{\sigma^2} \lesssim  \epsilon^{d/2}, \quad \quad \quad \frac{\kappa_3}{\sigma^2} \lesssim  \epsilon^{d/2} \log|\epsilon|.
\]
Proposition \ref{p:gauss} now follows immediately from Chatterjee's result. 
}
\vspace{0.2in}

\begin{proof}[Proof of Lemma \ref{lem:xesum}] If $|e| > 2 \epsilon^{-1}$ and $|x| \leq \epsilon^{-1}$, then $|x - e| > |e|/2 \geq 1$. So, clearly
\[
\sum_{|x| \leq \epsilon^{-1}} \frac{1}{(1 + |x - e|)^{d - 1}} \lesssim \frac{\epsilon^{-d}}{|e|^{d-1}} = \frac{\epsilon^{-1}}{|\epsilon e|^{d-1}} 
\]
in this case. For $|e| < 2 \epsilon^{-1}$, 
\[
 \sum_{|x| \leq \epsilon^{-1}} \frac{1}{(1 + |x - e|)^{d - 1}} \leq \int_{|y| < 3 \epsilon^{-1}} \frac{1}{(1 + |y|)^{d-1}} \,dy \lesssim \epsilon^{-1}.
\]
So, (\ref{xesum}) holds in this case, as well. 
\end{proof}

\vspace{0.2in}

\begin{proof}[Proof of Lemma \ref{lem:eepsum}] First, consider the sum over edges satisfying $|e| < \epsilon^{-1}$:
\br
\sum_{|e| < \epsilon^{-1}} \frac{1}{(1 + |e - e'|)^{d}} \frac{1}{(1 + |\epsilon e|)^p} & \lesssim & \sum_{|e| < \epsilon^{-1}} \frac{1}{(1 + |e - e'|)^{d}} \no \\
& \lesssim & \left \{\begin{array}{ll} |\log \epsilon|, & \forall\;e' \\ \frac{\epsilon^{-d}}{|e'|^d}, & \forall\; |e'| > \epsilon^{-1} \end{array} \right. \no \\
& \lesssim & \frac{|\log \epsilon|}{(1 + |\epsilon e'|)^d}.
\er
Next, consider the sum over edges satisfying $|e| > \epsilon^{-1}$:

\br
\sum_{|e| \geq \epsilon^{-1}} \frac{1}{(1 + |e - e'|)^{d}} \frac{1}{(1 + |\epsilon e|)^p} & \lesssim & \sum_{|e| \geq \epsilon^{-1}} \frac{1}{(1 + |e - e'|)^{d}} \frac{1}{|\epsilon e|^p}\no \\
& \lesssim & \epsilon^{-p} \int_{|x| > \epsilon^{-1}} \frac{1}{(1 + |x - e'|)^d} \frac{1}{|x|^p} \,dx.
\er
Let $x = \epsilon^{-1} z$, $e' = \epsilon^{-1} w$. Then
\br
\sum_{|e| \geq \epsilon^{-1}} \frac{1}{(1 + |e - e'|)^{d}} \frac{1}{(1 + |\epsilon e|)^p} & \lesssim &  \epsilon^{-d} \int_{|z| > 1} \frac{1}{(1 + \epsilon^{-d} |z - w|^d)} \frac{1}{|z|^p} \,dz. \label{xdpint}
\er
Restricting the integral in (\ref{xdpint}) to the set $|z - w| \leq \epsilon$, we have
\[ 
\epsilon^{-d} \int_{\substack{|z| > 1 \\ |z - w| \leq \epsilon}} \frac{1}{(1 + \epsilon^{-d} |z - w|^d)} \frac{1}{|z|^p} \,dz \lesssim  \epsilon^{-d} \int_{\substack{|z| > 1 \\ |z - w| \leq \epsilon}}  \frac{1}{|z|^p} \,dz \lesssim \min(1,|w|^{-p}).
\]
Restricting the integral in (\ref{xdpint}) to the set $|z - w| \geq \epsilon$, we have
\[ 
\epsilon^{-d} \int_{\substack{|z| > 1 \\ |z - w| \geq \epsilon}} \frac{1}{(1 + \epsilon^{-d} |z - w|^d)} \frac{1}{|z|^p} \,dz \lesssim   \int_{\substack{|z| > 1 \\ |z - w| \geq \epsilon}} \frac{1}{ |z - w|^d} \frac{1}{|z|^p} \,dz .
\]
If $|z| > 3|w|$, then $|w - z| > |z|/2$, so
\[
 \int_{\substack{|z| > 1 \\ |z - w| \geq \epsilon \\ |z| > 3|w|}} \frac{1}{ |z - w|^d} \frac{1}{|z|^p} \,dz \leq   \int_{\substack{|z| > 1 \\ |z| > 3|w|}} \frac{1}{|z|^{p+d}} \,dz \lesssim \min(1,|w|^{-p}).
\]
On the other hand, restricting to $|z| < 3|w|$, 
\br
\int_{\substack{|z| > 1 \\ |z - w| \geq \epsilon \\ |z| < 3|w|}} \frac{1}{ |z - w|^d} \frac{1}{|z|^p} \,dz & = & \int_{\substack{|z| > 1 \\ \epsilon \leq |z - w| \leq 1 \\ |z| < 3|w|}} \frac{1}{ |z - w|^d} \frac{1}{|z|^p} \,dz + \int_{\substack{|z| > 1 \\ |z - w| \geq 1 \\ |z| < 3|w|}} \frac{1}{ |z - w|^d} \frac{1}{|z|^p} \,dz\no \\
& \lesssim &  |\log \epsilon|  \min(1,|w|^{-p}) + \log (1 + |w|)\min(1,|w|^{-p}).
\er
The bound (\ref{eepsum}) now follows by combining these estimates with $w = \epsilon e'$. 
\end{proof}

\section{Proof of Proposition \ref{p:tight}}  \label{sec:tight}

By \cite{besov}, it suffices to show the following result.

\begin{prop} \label{prop:moments} 
For $f \in  C^0_c(\Rm^d;\Rm)$, let $f_\lambda(x) = \lambda^{-d} f(x/\lambda)$. 
For all $p \ge 1$, there exists a constant $C = C(p,f)$ such that for all $\epsilon, \lambda \in (0,1]$,
\[
 \avg{|\Phi_\epsilon(f_\lambda)|^p}^{1/p} \leq C \lambda^{1 - \frac{d}{2}} .
\]
\end{prop}
\begin{proof}
Observe that for any $\epsilon, \lambda \in (0,1]$,
\[
\Phi_\epsilon(f_\lambda) = \lambda^{ 1 - \frac{d}{2}} \Phi_r(f) 
\]
with $r = \epsilon/\lambda$.  Therefore, to prove Proposition \ref{prop:moments}, it suffices to show that
\be
 \avg{|\Phi_\epsilon(f)|^p}^{1/p} \leq C  \label{pmoment}
\ee
holds for all $\epsilon > 0$.

Since $\phi$ is stationary and $\avg{\phi(0)} = 0$, we know that $\avg{ \Phi_\epsilon(f)} = 0$. In particular,
\be
\avg{ \Phi_\epsilon(f)^{p} }^2 \lesssim 1  \quad  \text{ uniformly over }\epsilon > 0 \label{induct1}
\ee
holds for $p = 1$. Arguing inductively, let us suppose that (\ref{induct1}) holds for some positive integer $p = n$. We claim that (\ref{induct1}) must also hold for $p = 2n$.  To prove this claim, we use the identity
\[
\avg{ \Phi_\epsilon(f)^{2n} } = \avg{ \Phi_\epsilon(f)^{n} }^{2} + \var(\Phi_\epsilon(f)^n),
\]
which, by the induction hypothesis, gives us
\be
\avg{ \Phi_\epsilon(f)^{2n} } \lesssim 1 + \var(\Phi_\epsilon(f)^n). \label{induct2}
\ee
We will show that $\var(\Phi_\epsilon(f)^n) \lesssim \avg{ \Phi_\epsilon(f)^{2n} }^{1 - 1/n}$, hence
\be
\avg{ \Phi_\epsilon(f)^{2n} } \lesssim 1 + \avg{ \Phi_\epsilon(f)^{2n} }^{1 - 1/n}. \label{induct3}
\ee
This bound and Young's inequality establish the claim that (\ref{induct1}) also holds for $p = 2n$.

To prove that $\var(\Phi_\epsilon(f)^n) \lesssim \avg{ \Phi_\epsilon(f)^{2n} }^{1 - 1/n}$, we apply the spectral gap inequality and H\"older's inequality:
\br
\var(\Phi_\epsilon(f)^n) & \lesssim & \sum_{e} \avg{| \partial_e \Phi_\epsilon(f)^n |^2} \no \\
& \lesssim & \sum_{e} \avg{| \Phi_\epsilon(f)^{n-1}  \partial_e \Phi_\epsilon(f) |^2}  \no \\
& \leq & \sum_{e} \avg{ \Phi_\epsilon(f)^{2n}}^{1 - 1/n}   \avg{|\partial_e \Phi_\epsilon(f)|^{2n}}^{1/n} \no \\
& = &  \avg{ \Phi_\epsilon(f)^{2n}}^{1 - 1/n} \sum_{e}   \avg{|\partial_e \Phi_\epsilon(f)|^{2n}}^{1/n}. \label{varnupper}
\er
Now we show that the last sum is bounded by a constant. Without loss of generality, assume that $f(x) = 0$ for $|x| > 1$. Recall that $\bigavg{ |\partial_e \phi(x)|^p }^{1/p}  \lesssim (1 + |x - e|)^{1 - d}$ (see (\ref{dephi})). So,
\br
 \bigavg{ \left( \partial_e \Phi(f) \right)^{2n} }^{1/(2n)}  & = & \epsilon^{\frac{d}{2} + 1}    \bigavg{  \left(\sum_{x} f(\epsilon x) \partial_e \phi(x) \right)^{2n}}^{1/(2n)} \no \\
& \lesssim & \epsilon^{\frac{d}{2} + 1}  \sum_{|x| \le \epsilon^{-1} }  \avg{ \left(\partial_e \phi(x) \right)^{2n} }^{1/(2n)} \no \\
& \lesssim & \epsilon^{\frac{d}{2} + 1}  \sum_{|x| \le \epsilon^{-1} } \frac{1}{(1 + |x - e|)^{d - 1}}. \no 
\er
The last sum is controlled by Lemma \ref{lem:xesum}, which leads to
\[
\bigavg{ \left( \partial_e \Phi(f) \right)^{2n} }^{1/(2n)}  \leq \epsilon^{\frac{d}{2} + 1} \frac{\epsilon^{-1}}{(1 + |\epsilon e|)^{d-1}}.
\]
Therefore,
\br
\sum_{e}   \avg{|\partial_e \Phi_\epsilon(f)|^{2n}}^{1/n} \lesssim \epsilon^{d + 2} \sum_e  \frac{\epsilon^{-2}}{(1 + |\epsilon e|)^{2d-2}} \lesssim 1
\er
since $d \geq 3$. In view (\ref{varnupper}), we have now established that $\var(\Phi_\epsilon(f)^n) \lesssim \avg{ \Phi_\epsilon(f)^{2n} }^{1 - 1/n}$. 

By induction on $n$, we have proved that
\[
\avg{ \Phi_\epsilon(f)^{p} }^2 \lesssim 1
\]
holds for all $p \in \{ 2^n \;|\; n = 0,1,2,3, \dots \}$. Now (\ref{pmoment}) follows by Jensen's inequality.
\end{proof}

\section{Stein's method} \label{sec:stein}

In this section, we prove Proposition \ref{prop:stein}, which is a version of Theorem~2.2 in \cite{Ch3} and Theorem~3.1 (and Remark 3.6) in \cite{NP1}, stated in a form that is convenient for our purpose. The basis of the estimate is the following lemma:

\begin{lem} [See \cite{Chat1}, Lemma 4.2] \label{lem:stein}
Suppose $h:\Rm \to \Rm$ is absolutely continuous with bounded derivative, and $Y \sim N(0,1)$. There exists a solution to
\be
\phi'(x) - x \phi(x) = h(x) - \avg{h(Y)}, \quad x \in \Rm \label{steineqn}
\ee
which satisfies $\norm{\phi'}_\infty \leq \sqrt{\frac{2}{\pi}} \; \norm{h'}_\infty$ and $\norm{\phi''}_{\infty} \leq 2 \norm{h'}_\infty$. 
\end{lem}

\begin{proof}[Proof of Proposition \ref{prop:stein}]

Let $F \in L^2(\Omega)$ be such that $\avg{F} = 0$, $\avg{F^2} = 1$. By (\ref{steineqn}), we have
\br
\left \langle  h(F) - h(Y) \right \rangle & = & \left\langle \phi'(F) - \phi(F) F \right \rangle \no \\
& = & \text{Cov}( \avg{\phi'(F)} F - \phi(F) , F). \no
\er
Since $\partial_e F \in L^2(\Omega)$ for all $e \in \mathbb{B}$, we can estimate this covariance by applying the Helffer-Sj\"ostrand correlation representation (see \cite{correl}, Prop.\ 3.1) to $\text{Cov}(G,F)$ with $G =  \avg{\phi'(F)} F - \phi(F)$. Let $\partial_e^* = - \partial_e + \zeta_e$ be the adjoint of the derivative operator $\partial_e$.  Let $\mathscr{L} = \partial^* \partial$, where $\partial F = (\partial_e F)_{e \in \mathbb{B}}$ and for $K = (K_e)_{e \in \mathbb{B}}$, $\partial^* K = \sum_e \partial_e^* K_e$.  From the correlation representation, we obtain
\br
\avg{h(F) - h(Y)} & = & \sum_{e} \bigavg{ \partial_e G (\mathscr{L} + 1)^{-1} \partial_e F} \no \\
& = & \sum_{e} \bigavg{ \left( \avg{\phi'(F)} - \phi'(F) \right)\partial_e F (\mathscr{L} + 1)^{-1} \partial_e F}.
\er
Observe that $\avg{\phi'(F)} - \phi'(F) \in L^\infty$ has zero mean. Moreover, the series $\sum_e \partial_e F (\mathscr{L} + 1)^{-1} \partial_e F$ converges in $L^1(\Omega)$ and has mean $1 = \text{Cov}(F,F)$. Hence
\br
\avg{h(F) - h(Y)} & = &  \bigavg{ \left( \avg{\phi'(F)} - \phi'(F) \right)\left(\sum_{e} \partial_e F (\mathscr{L} + 1)^{-1} \partial_e F \right) } \no \\
& = &  \bigavg{ \phi'(F) \left(1 - \sum_{e} \partial_e F (\mathscr{L} + 1)^{-1} \partial_e F \right) } \no \\
& \leq & \norm{\phi'}_\infty \bigavg{ \left|1 - \sum_{e} \partial_e F(\mathscr{L} + 1)^{-1} \partial_e F  \right| } \no \\
& \leq & \sqrt{\frac{2}{\pi}} \norm{h'}_\infty \bigavg{ \left|1 - \sum_{e} \partial_e F (\mathscr{L} + 1)^{-1} \partial_e F \right| }. \label{L1sum}
\er

Assuming that $\sum_{e} \avg{|\partial_e F|^4}^{1/2} < \infty$, we claim that the sum 
\[
S = \sum_{e} \partial_e F (\mathscr{L} + 1)^{-1} \partial_e F
\]
is in $L^2(\Omega)$, and we will estimate the last term in (\ref{L1sum}) by the variance
\be
\bigavg{ \left|1 - \sum_{e} \partial_e F (\mathscr{L} + 1)^{-1} \partial_e F  \right| } \leq \sqrt{\var \left(\sum_{e} \partial_e F (\mathscr{L} + 1)^{-1} \partial_e F \right)}. \label{Gspectral}
\ee
Indeed,
\[
\norm{ \partial_e F (\mathscr{L} + 1)^{-1} \partial_e F}_2 \leq \norm{ \partial_e F}_4 \norm{ (\mathscr{L} + 1)^{-1} \partial_e F}_4 \leq \norm{ \partial_e F}_4^2,  
\]
since $(\mathscr{L} + 1)^{-1}$ is a contraction on $L^p(\Omega)$ for any $p \geq 2$ (see Prop.\ 3.2 of \cite{correl}). Thus, $S \in L^2(\Omega)$ if $\sum_{e} \norm{ \partial_e F}_4^2 < \infty$.

Under the additional assumption $\avg{|\partial_{e'} \partial_e F|^4} < \infty$ for all $e,e' \in \mathbb{B}$, we also have $\partial_{e'} S \in L^2(\Omega)$ for all $e' \in \mathbb{B}$, so that $\var(S)$ can be estimated by the Gaussian spectral gap inequality. We compute: 
\[
\partial_{e'}(\mathscr{L} + 1)^{-1} \partial_e F = (\mathscr{L} + 2)^{-1} \partial_{e'} \partial_e F,
\]
so that
\[
\partial_{e'} (\partial_e F (\mathscr{L} + 1)^{-1} \partial_e F) = (\partial_{e'} \partial_e F) (\mathscr{L} + 1)^{-1} \partial_e F + \partial_e F (\mathscr{L} + 2)^{-1} \partial_{e'} \partial_e F.
\]
The operator $(\mathscr{L} + 2)^{-1}$ is also a contraction, satisfying $\norm{(\mathscr{L} + 2)^{-1} u}_p \leq (1/2) \norm{u}_p$ for all $u \in L^p(\Omega)$, $p \geq 2$ (this follows from minor modification to the proof of Prop.\ 3.2 of \cite{correl}). Therefore, 
\br
\norm{ \partial_{e'} (\partial_e F (\mathscr{L} + 1)^{-1} \partial_e F)}_2 & \leq & \norm{(\partial_{e'} \partial_e F) (\mathscr{L} + 1)^{-1} \partial_e F}_2 + \norm{\partial_e F (\mathscr{L} + 2)^{-1} \partial_{e'} \partial_e F}_2 \no \\
& \leq & \norm{\partial_{e'} \partial_e F}_4 \norm{ (\mathscr{L} + 1)^{-1} \partial_e F}_4 + \norm{\partial_e F}_4 \norm{(\mathscr{L} + 2)^{-1} \partial_{e'} \partial_e F}_4 \no \\
& \leq & (1 + \frac{1}{2}) \norm{\partial_{e'} \partial_e F}_4 \norm{  \partial_e F}_4 < \infty.
\label{e:CS}
\er
Hence, $\partial_{e'} S \in L^2(\Omega)$ for all $e' \in \mathbb{B}$.

Now we apply the Gaussian spectral gap inequality:
\br
\var\left(\sum_{e} \partial_e F (\mathscr{L} + 1)^{-1} \partial_e F \right) & \leq & \sum_{e'} \left \langle \left| \partial_{e'} \sum_e \partial_e F (\mathscr{L} + 1)^{-1} \partial_e F \right|^2  \right \rangle \no \\
& \leq & 2 \sum_{e'} \left \langle \left|  \sum_e (\partial_{e'} \partial_e F) (\mathscr{L} + 1)^{-1} \partial_e F \right|^2 \right \rangle \no \\
& & + 2 \sum_{e'} \left \langle \left|  \sum_e \partial_e F (\mathscr{L} + 2)^{-1} \partial_{e'} \partial_e F \right|^2 \right \rangle. \label{varsum}
\er
To bound the last two sums, we apply Minkowski's inequality:
\br
 \left \langle \left|  \sum_e (\partial_{e'} \partial_e F) (\mathscr{L} + 1)^{-1} \partial_e F \right|^2 \right \rangle^{1/2} & \leq &  \sum_e  \left \langle \left| (\partial_{e'} \partial_e F) (\mathscr{L} + 1)^{-1} \partial_e F \right|^2 \right \rangle^{1/2}, \no \\
 \left \langle \left|  \sum_e (\partial_{e'} \partial_e F) (\mathscr{L} + 1)^{-1} \partial_e F \right|^2 \right \rangle^{1/2} & \leq &  \sum_e  \left \langle \left| \partial_{e}F (\mathscr{L} + 2)^{-1} \partial_{e'} \partial_e F \right|^2 \right \rangle^{1/2}. \no
\er
As was seen in \eqref{e:CS}, we have
\br
\left \langle \left| (\partial_{e'} \partial_e F) (\mathscr{L} + 1)^{-1} \partial_e F \right|^2 \right \rangle^{1/2} 
 & \leq & \left \langle | \partial_{e'} \partial_e F|^4 \right \rangle^{1/4} \left \langle |\partial_e F |^4 \right \rangle^{1/4}, \no \\
\left \langle \left| \partial_{e}F (\mathscr{L} + 2)^{-1} \partial_{e'} \partial_e F \right|^2 \right \rangle^{1/2} 
 & \leq &  \frac{1}{2} \left \langle |\partial_e F |^4 \right \rangle^{1/4}  \left \langle | \partial_{e'} \partial_e F|^4 \right \rangle^{1/4}.
\er
Therefore, returning to (\ref{varsum}), we obtain that
\br
\var\left(\sum_{e} \partial_e F (\mathscr{L} + 1)^{-1} \partial_e F \right) \leq \frac{5}{2} \sum_{e'} \left( \sum_{e} \left \langle |\partial_e F |^4 \right \rangle^{1/4} \left \langle | \partial_{e'} \partial_e F|^4 \right \rangle^{1/4} \right)^2 , \no
\er
and thus from (\ref{L1sum}) and (\ref{Gspectral}) we conclude
\be
\avg{h(F) - h(Y)}  \leq  \sqrt{\frac{5}{\pi}} \norm{h'}_\infty \sqrt{\sum_{e'} \left( \sum_{e} \left \langle |\partial_e F |^4 \right \rangle^{1/4} \left \langle | \partial_{e'} \partial_e F|^4 \right \rangle^{1/4} \right)^2}. \label{steinsum2}
\ee
(The assumptions on $F$ in Proposition \ref{prop:stein} do not guarantee that the right side of (\ref{steinsum2}) is finite. In the case that the right side is infinite, the conclusion of the theorem holds trivially.)
\end{proof}


\section*{Acknowledgements}
JN acknowledges financial support from US National Science Foundation grants DMS-1007572 and DMS-1351653.

\end{document}